\documentclass[12pt]{article}

\usepackage{latexsym, amssymb,amsmath,amscd, amsfonts, amsthm, ifpdf, fancybox,float,  geometry, epsfig, graphicx, color, colordvi, setspace,caption,subcaption,indentfirst,soul}

\allowdisplaybreaks[4]
\newtheorem{thm}{Theorem}[section]

\newtheorem{lem}[thm]{Lemma}
\newtheorem{core}[thm]{Corollary}

\def\Sg{{\rm Sg}}

\numberwithin{equation}{section}

\sethlcolor{yellow}
\soulregister\cite7
\soulregister\ref7 
\soulregister\eqref7 

\begin{document}
	
\begin{center}
	
	{\Large \bf Bilateral truncated quintuple product}
\end{center}

\begin{center}
	{  Wenxia Qu}$^{1}$,  and
	{Wenston J.T. Zang}$^{2}$ \vskip 2mm
	
	$^{1,2}$School of Mathematics and Statistics, Northwestern Polytechnical University, Xi'an 710072, P.R. China\\[6pt]
	$^{1,2}$ MOE Key Laboratory for Complexity Science in Aerospace, Northwestern Polytechnical University, Xi'an 710072, P.R. China\\[6pt]
	$^{1,2}$ Xi'an-Budapest Joint Research Center for Combinatorics, Northwestern Polytechnical University, Xi'an 710072, P.R. China\\[6pt]
	\vskip 2mm
	
	$^1$quwenxia0710@mail.nwpu.edu.cn, $^2$zang@nwpu.edu.cn
\end{center}

\vskip 6mm \noindent {\bf Abstract.}  In this paper, we present the bilateral truncated identity of the quintuple product identity, which is a generalization of the truncated quintuple product identities given by Chan, Ho and Mao [J. Number Theory 169 (2016) 420--438]. Additionally, we provide the bilateral truncated forms of two $q$-series identities, which are well-known consequences of the quintuple product identity.



\section{Introduction}

The quintuple product identity plays a crucial role in  $q$-series. This identity was discovered about 110 years ago, and it has at least 29 different proofs. These proofs involve many different approaches, including basic hypergeometric series, elliptic functions, sigma functions and so on.  See \cite{cooper2006quintuple} for more details.

Andrews and Merca \cite{andrews2012truncated} initiated the study of the truncated $q$-series identity by giving the truncated Euler pentagonal number theorem.  Guo and Zeng \cite{guo2013two} gave two truncated Gauss identities. The truncated Jacobi triple product was established by Mao \cite{mao2015proofs} and Yee \cite{yee2015truncated} independently (see Wang and Yee \cite{wang2019truncated} for an explicit form,  Merca \cite{merca2021truncatedtheta} for a stronger conjecture and \cite{Ballantine-FeigonTruncated, Ding-SunTruncated, Ding-SunProof} for some partial result of this conjecture). In addition, there were some studies related to bilateral truncated identities. For example, He, Ji and Zang \cite{he2016bilateral} gave bilateral truncated Jacobi identities, and Li \cite{li2023generalized} established a bilateral truncated Jacobi triple product. We remark that the truncated identities have been recently studied in several papers by  Andrews and Merca \cite{andrews2018truncated}, Chern and Xia \cite{chern2024two}, Merca \cite{merca2022two},  Wang and Yee \cite{wang2020truncated, wang2023truncated}, Xia, Yee and Zhao \cite{xia2022new} and Zhou \cite{zhou2024positivity}.

This paper focuses on the bilateral truncated form of the quintuple product identity. Recall that the quintuple product identity usually has the following form \cite[pp. 118]{cooper2006quintuple},
\begin{equation}\label{cooper-quituple-product}
	\sum_{n=-\infty}^{\infty}q^{n(3n+1)/2}(x^{3n}-x^{-3n-1})=(x^{-1},xq,q;q)_{\infty}(x^{-2}q,x^2q;q^2)_{\infty}.
\end{equation}
Here we use the standard $q$-series notation
\[(a;q)_n=\prod_{i=1}^{n}(1-aq^{i-1}),\quad
(a;q)_\infty=\prod_{i=1}^{\infty}(1-aq^{i-1})\]
and
\[(a_1,a_2,\ldots,a_k;q)_n=\prod_{i=1}^k(a_i;q)_n,\quad (a_1,a_2,\ldots,a_k;q)_\infty=\prod_{i=1}^k(a_i;q)_\infty.\]

The following equation is an alternating form of the quintuple product identity by substituting $x$ and $q$ into $q^S$ and $q^R$ in \eqref{cooper-quituple-product} respectively.
\begin{equation}\label{eq-qui-trans}	
	\sum_{n=-\infty}^{\infty}q^{n(3n+1)R/2}(q^{3nS}-q^{-(3n+1)S})=(q^{-S},q^{R+S},q^{R};q^{R})_{\infty}(q^{R-2S},q^{R+2S};q^{2R})_{\infty}.
\end{equation}

Chan, Ho and Mao \cite{chan2016truncated} found the following two truncated forms of \eqref{eq-qui-trans}.

\begin{thm}[\cite{chan2016truncated}]\label{chan2016truncated}
	For $1\leq S < R/2$ and $k\geq 0$, the truncated series
	\begin{equation}\label{tru-CHM-qu1}
		\frac{1}{(q^{-S},q^{R+S},q^R;q^R)_{\infty}(q^{R-2S},q^{R+2S};q^{2R})_{\infty}}\sum_{n=-k}^{k}q^{n(3n+1)R/2}(q^{3nS}-q^{-(3n+1)S})
	\end{equation}
	has non-negative coefficients.
\end{thm}
\begin{thm}[\cite{chan2016truncated}]\label{chan-truncated -ktok-1}
	For $1\leq S < R/2$ and $k\geq 1$ , the truncated series
	\begin{equation}\label{tru-CHM-qu2}
		-1+\frac{1}{(q^{-S},q^{R+S},q^R;q^R)_{\infty}(q^{R-2S},q^{R+2S};q^{2R})_{\infty}}\sum_{n=-k}^{k-1}q^{n(3n+1)R/2}(q^{3nS}-q^{-(3n+1)S})
	\end{equation}
	has non-positive coefficients.
\end{thm}

Let
\begin{equation}\notag
	\Sg(x):=
	\begin{cases}
		1, & \text{if }x \geq 0\\
		-1, & \text{if }x < 0
	\end{cases}.
\end{equation}
The first main result of this paper is to give a bilateral generalization of both Theorem \ref{chan2016truncated} and Theorem \ref{chan-truncated -ktok-1} as follows.

\begin{thm}\label{thm-2}
	For any	 $a,b\in \mathbb{Z}$ and positive integers $S$,$R$ with $1\leq S < R/2$,
	\begin{equation}\label{thm2-eq}
		\frac{\Sg(a+b)}{(q^{-S},q^{R+S},q^R;q^R)_{\infty}(q^{R-2S},q^{R+2S};q^{2R})_{\infty}}\sum_{n=a}^{b}q^{n(3n+1)R/2}(q^{3nS}-q^{-(3n+1)S})
	\end{equation}
	has non-negative coefficients of $q^k$ for any $k\ge 1$.
\end{thm}

The second goal of this paper is to provide bilateral truncated generalizations of the following two identities, which are well-known as the consequences  of the quintuple product identity   (see \cite{ramanujan1916on,watson1931ramanujan}).
\begin{equation}\label{quituple3n+1} \sum_{n=-\infty}^{\infty}(3n+1)q^{3n^2+2n}=(q;q^2)_{\infty}^{2}(q^2;q^2)_{\infty}(q^4;q^4)_{\infty}^2,
\end{equation}
\begin{equation}\label{quituple6n+1}
	\sum_{n=-\infty}^{\infty}(6n+1)q^{3n^2+n}=(q^2;q^2)_{\infty}^{3}(q^2;q^4)_{\infty}^{2}.
\end{equation}
Chan, Ho and Mao \cite{chan2016truncated} also found the truncated forms of the above two identities as follows.

\begin{thm}[\cite{chan2016truncated}]\label{chan2016truncated-1}
	For $k\geq 0$, the truncated series
	\begin{equation}\label{chan-truncated-(3n+1)}
		\frac{1}{(q;q^2)_{\infty}^2(q^2;q^2)_{\infty}(q^4;q^4)_{\infty}^2}\sum_{n=-k}^{k}(3n+1)q^{3n^2+2n}
	\end{equation}
	has non-negative coefficients.
\end{thm}
\begin{thm}[\cite{chan2016truncated}]\label{thmchan2016truncated6n+1}
	For $k\geq 0$, the truncated series
	\begin{equation}\label{tru-CHM-qu3}
		\frac{1}{(q^2;q^2)_{\infty}^{3}(q^2;q^4)_{\infty}^{2}}\sum_{n=-k}^{k}(6n+1)q^{3n^2+n}
	\end{equation}
	has non-negative coefficients.
\end{thm}

We find the following bilateral truncated form  of Theorem \ref{chan2016truncated-1}.

\begin{thm}\label{thm-3n+1}
	For $a,b\in \mathbb{Z}$, the truncated series
	\begin{equation}\label{3n+1atobmainthm} \frac{\Sg(a+b)}{(q;q^2)_{\infty}^2(q^2;q^2)_{\infty}(q^4;q^4)_{\infty}^2}\sum_{n=a}^{b}(3n+1)q^{3n^2+2n}
	\end{equation}
	has non-negative coefficients of $q^k$ for any $k\ge 1$.
\end{thm}

Similarly, the following theorem is a  bilateral generalization of Theorem \ref{thmchan2016truncated6n+1}.
\begin{thm}\label{thm-3}
	For $a,b\in \mathbb{Z}$, the truncated series
	\begin{equation}\label{eqatob}  		\frac{\Sg(a+b)}{(q^2;q^2)_{\infty}^{3}(q^2;q^4)_{\infty}^{2}}\sum_{n=a}^{b}(6n+1)q^{3n^2+n}
	\end{equation}
	has non-negative coefficients of $q^k$ for any $k\ge 1$.
\end{thm}

Recall that the overpartition of $n$ was defined by Corteel and Lovejoy \cite{corteel2004overpartitions} as a non-increasing sequence of natural number   in which the first occurrence of a number maybe overlined. Let $\overline{p}(n)$ denote the number of overpartitions of $n$. For instance, there are $8$ overpartitions of $3$, namely,
\begin{equation}\notag
	(3),(\overline{3}),(2,1),(\overline{2},1),(2,\overline{1}),(\overline{2},\overline{1}),(1,1,1),(\overline{1},1,1).
\end{equation}
Hence $\overline{p}(n)=8$.

Chan, Ho and Mao \cite{chan2016truncated} introduced the function $p\overline{pp}(n)$ as the number of partition triplets $(\lambda,\alpha,\beta)$, where $\lambda$ is an ordinary partition,  $\alpha,\beta$ are both overpartitions and $|\lambda|+|\alpha|+|\beta|=n$. For example, $p\overline{pp}(2)=18$, and the $18$ triplets  are listed as follows.
\[
	\begin{array}{llllll}
		((2),\emptyset,\emptyset),&((1,1),\emptyset,\emptyset),&(\emptyset,(2),\emptyset),&(\emptyset,(\overline{2}),\emptyset),&(\emptyset,(1,1),\emptyset),&(\emptyset,(\overline{1},1),\emptyset),\\
		(\emptyset,\emptyset,(2)),&(\emptyset,\emptyset,(\overline{2})),&(\emptyset,\emptyset,(1,1)),&(\emptyset,\emptyset,(\overline{1},1)),&((1),(1),\emptyset),&((1),(\overline{1}),\emptyset),\\
		((1),\emptyset,(1)),&((1),\emptyset,(\overline{1})),&(\emptyset,(1),(1)),&(\emptyset,(\overline{1}),(1)),&(\emptyset,(1),(\overline{1})),&(\emptyset,(\overline{1}),(\overline{1})).
	\end{array}
\]

By Theorem \ref{thm-3},   the following corollary is clear.

\begin{core}\label{cor}
	For any $n\ge 0$ and $a,b\in \mathbb{Z}$,
	\[\Sg(a+b)\sum_{i=a}^{b}(6i+1)p\overline{pp}(n-\frac{i(3i+1)}{2})\ge 0.\]
\end{core}

This paper is organized as follows. In Section \ref{sectionthm2}, we first prove Lemma \ref{lemqn} which is the key ingredient of our proofs. Then we demonstrate that Theorem \ref{thm-2} is a direct consequence of Lemma \ref{lemqn}, Theorem \ref{chan2016truncated} and Theorem \ref{chan-truncated -ktok-1}. Section \ref{sectionthm-3n+1} is devoted to giving a proof of Theorem \ref{thm-3n+1}. To this end, we will show that \eqref{3n+1atobmainthm} satisfies the two restrictions in Lemma \ref{lemqn}. In fact, the first restriction is trivial to verify. Moreover, the first part of the second restriction will be proved in Lemma \ref{lem3n+1-6} and the second part of the second restriction coincides with Theorem \ref{chan2016truncated-1}. This leads to a proof of Theorem \ref{thm-3n+1}. The proof of Theorem \ref{thm-3} will be given in Section \ref{sectionthm3} in the same way. We will show that Theorem \ref{thm-3} is a consequence of Lemma \ref{lemqn}, Lemma \ref{lem6n+1-5}, and Theorem \ref{thmchan2016truncated6n+1}.   Corollary \ref{cor} will also be proved in Section \ref{sectionthm3}.

\section{Proof of Theorem \ref{thm-2}}\label{sectionthm2}

In this section, we first prove the following lemma, which will be used in the proof of Theorem \ref{thm-2}, Theorem \ref{thm-3n+1} and Theorem \ref{thm-3}. We then give a proof of Theorem \ref{thm-2} with the aid of Lemma \ref{lemqn}.

\begin{lem}\label{lemqn}
	Let $f\colon \mathbb{Z}\rightarrow \mathbb{Z}$ satisfy the following two restrictions.
	\begin{itemize}
		\item[(1)] For any $n\in\mathbb{Z}$,
		\begin{equation}\label{equ-sgn-fn}
			\Sg(n)f(n)\ge 0.
		\end{equation}
		\item[(2)]For any $k\ge 1$,
		\begin{equation}\label{equ-sum-n-k}
			\sum_{n=-k}^{k-1}f(n)\le 0\quad\text{and}\quad\sum_{n=-k}^{k}f(n)\ge 0.
		\end{equation}
	\end{itemize} 
	 Then  we have
	  \begin{equation}\label{sga+b}
	  	\Sg(a+b)\sum_{n=a}^{b}f(n)\ge 0
	  \end{equation}
	  for any $ a, b\in \mathbb{Z}$.
\end{lem}
\begin{proof}
	We consider the following five cases.
	\begin{itemize}
		\item[\textbf{Case 1.}] $a>b$. Clearly in this case $\sum_{n=a}^{b}f(n)=0$. Thus \eqref{sga+b} holds.
		\item[\textbf{Case 2.}] $b\ge a\ge 0$.
		In this case, it is trivial to see that $\Sg(a+b)=1$. Moreover, by \eqref{equ-sgn-fn} we see that for $a\le n\le b$, $f(n)\ge 0$. Thus in this case, all the $f(n)$ in $\Sg(a+b)\sum_{n=a}^{b}f(n)$ are non-negative. This yields \eqref{sga+b} holds.
		\item[\textbf{Case 3.}] $a\le b<0$.
		In this case, we have $\Sg(a+b)=-1$. Moreover,   \eqref{equ-sgn-fn} implies that $f(n)\le 0$ for $a\le n\le b$. Thus  all the $f(n)$ in $\sum_{n=a}^{b}f(n)$ are non-positive. Together with $\Sg(a+b)=-1$, we derive \eqref{sga+b}.
		\item[\textbf{Case 4.}] $a<0<b$ and $b \geq -a$.
		In this case, we see that $\Sg(a+b)=1$. Moreover,
		\begin{equation}\label{eq-case4-lem}
			\sum_{n=a}^{b}f(n)=\sum_{n=a}^{-a}f(n)+\sum_{n=-a+1}^{b}f(n).
		\end{equation}
		By \eqref{equ-sum-n-k}, we see that $\sum_{n=a}^{-a}f(n)\ge 0$. Moreover, by \eqref{equ-sgn-fn} we have $f(n)\ge 0$ for $-a+1\le n\le b$, which implies $\sum_{n=-a+1}^{b}f(n)\ge 0$. From the above analysis, we have $\sum_{n=a}^{b}f(n)\ge 0$. Combining $\Sg(a+b)=1$,   we see that   \eqref{sga+b} holds.
		\item[\textbf{Case 5.}] $a<0\le b$ and $b < -a$.
		Note that $\Sg(a+b)=-1$ and
		\begin{equation}
			\sum_{n=a}^{b}f(n)=\sum_{n=a}^{-b-2}f(n)+\sum_{n=-b-1}^{b}f(n).\label{eqqnato-b-2-b-1tob}
		\end{equation}
		By \eqref{equ-sum-n-k}, we see that $\sum_{n=-b-1}^{b}f(n)\le 0$. Moreover, \eqref{equ-sgn-fn} suggests that  $f(n)\le 0$ for $a\le n\le -b-2$. Hence $\sum_{n=a}^{-b-2}f(n)\le 0$.  Together with $\Sg(a+b)=-1$,   we deduce that   \eqref{sga+b} holds.
	\end{itemize}
\end{proof}
We are now in a position to show Theorem \ref{thm-2}, which is a direct consequence of 	Theorem \ref{chan2016truncated}, Theorem \ref{chan-truncated -ktok-1} and Lemma \ref{lemqn}.

{\noindent \it Proof of Theorem \ref{thm-2}.}
For any $i\ge 0$, let $f_{i}(n)$ denote the coefficient of $q^i$ in
\[\frac{q^{n(3n+1)R/2}(q^{3nS}-q^{-(3n+1)S})}{(q^{-S},q^{R+S},q^R;q^R)_{\infty}(q^{R-2S},q^{R+2S};q^{2R})_{\infty}}.\]
In other words,
\begin{equation*}
	\frac{q^{n(3n+1)R/2}(q^{3nS}-q^{-(3n+1)S})}{(q^{-S},q^{R+S},q^R;q^R)_{\infty}(q^{R-2S},q^{R+2S};q^{2R})_{\infty}}=\sum_{i=0}^{\infty}f_{i}(n)q^i.
\end{equation*}
We proceed to verify that $f_i(n)$ satisfies the two restrictions in Lemma \ref{lemqn}. On the one hand,  when $n\ge 0$, it is easy to check that
\begin{align}\label{fracsga+b}
	& \frac{q^{n(3n+1)R/2}(q^{3nS}-q^{-(3n+1)S})}{(q^{-S},q^{R+S},q^R;q^R)_{\infty}(q^{R-2S},q^{R+2S};q^{2R})_{\infty}}\nonumber\\[3pt]
	=&\frac{q^{(3n+1)(nR/2-S)}(1-q^{(6n+1)S}) q^{S}}{(1-q^S)(q^{R-S},q^{R+S},q^R;q^R)_{\infty}(q^{R-2S},q^{R+2S};q^{2R})_{\infty}}\nonumber\\[3pt]
	=&\frac{q^{(3n+1)(nR/2-S)}(1+q^S+q^{2S}+\cdots+q^{6nS}) q^{S}}{(q^{R-S},q^{R+S},q^R;q^R)_{\infty}(q^{R-2S},q^{R+2S};q^{2R})_{\infty}}
\end{align}
which has non-negative coefficients. Thus $f_i(n)\ge 0$ for $n\ge 0$.

When $n<0$,
\begin{align}\label{fracsga+b1}
	& \frac{q^{n(3n+1)R/2}(q^{3nS}-q^{-(3n+1)S})}{(q^{-S},q^{R+S},q^R;q^R)_{\infty}(q^{R-2S},q^{R+2S};q^{2R})_{\infty}}\nonumber\\[3pt]
	=&\frac{q^{(3n+1)(nR/2-S)}(1-q^{(6n+1)S}) q^{S}}{(1-q^S)(q^{R-S},q^{R+S},q^R;q^R)_{\infty}(q^{R-2S},q^{R+2S};q^{2R})_{\infty}}\nonumber\\[3pt]
	=&\frac{q^{(3n+1)(nR/2-S)+(6n+2)S}}{(q^{R-S},q^{R+S},q^R;q^R)_{\infty}(q^{R-2S},q^{R+2S};q^{2R})_{\infty}}\frac{q^{(-6n-1)S}-1}{1-q^S}\notag\\
	=&-\frac{q^{(3n+1)(nR/2+S)}}{(q^{R-S},q^{R+S},q^R;q^R)_{\infty}(q^{R-2S},q^{R+2S};q^{2R})_{\infty}}(1+q^S+\cdots+q^{(-6n-2)S}).
\end{align}
Clearly \eqref{fracsga+b1} has non-positive coefficients. This means $f_i(n)\le 0$ for $n<0$.
So in either case, we have $\Sg(n)f_i(n)\ge 0$.

On the other hand,  Theorem  \ref{chan2016truncated} implies that
$\sum_{n=-k}^{k}f_i(n)\ge 0$ for any $i\ge 0$ and $k\ge 1$. Moreover, by Theorem \ref{chan-truncated -ktok-1}, we have $\sum_{n=-k}^{k-1}f_i(n)\le 0$ for any $ i\ge 1 $ and $k\ge 1$. Thus when $i\ge 1$, we show that $f_i(n)$ satisfies both \eqref{equ-sgn-fn} and \eqref{equ-sum-n-k}. So by Lemma \ref{lemqn}, we deduce $\Sg(a+b)\sum_{n=a}^{b}f_i(n)\ge 0$ for any $a,b\in \mathbb{Z}$ and  $i\ge 1$. This yields that the coefficients of $q^i$ in \eqref{thm2-eq} are non-negative for any $i\ge 1$.
\qed

\section{Proof of Theorem \ref{thm-3n+1}}\label{sectionthm-3n+1}
In this section, we give a proof of Theorem \ref{thm-3n+1} with the aid of Lemma \ref{lemqn}. To this end, let $g_i(n)$ denote the coefficient of $q^i$ in
\[	 \frac{(3n+1)q^{3n^2+2n}}{(q;q^2)_{\infty}^2(q^2;q^2)_{\infty}(q^4;q^4)_{\infty}^2}.
\]
In other words,
\begin{equation}\label{equ-sum-i=0}
	\sum_{i=0}^{\infty}g_i(n)q^i=\frac{(3n+1)q^{3n^2+2n}}{(q;q^2)_{\infty}^2(q^2;q^2)_{\infty}(q^4;q^4)_{\infty}^2}.\end{equation}
As stated in the Introduction, we proceed to verify that $g_i(n)$ satisfies the two {restrictions} in  Lemma \ref{lemqn}. It is clear that $g_i(n)$ meets \eqref{equ-sgn-fn}. We next verify that \eqref{equ-sum-n-k} also holds for $g_i(n)$. In fact, we have the following result.

\begin{lem}\label{lem3n+1-6}
	For $k \ge 1$ and $i\ge 1$,
	\begin{equation}\label{eqlem3n+1-ktok-1}
	\sum_{n=-k}^{k-1}g_i(n)\le 0.
	\end{equation}
\end{lem}

In order to show Lemma \ref{lem3n+1-6}, we define
\[
A:=\sum_{n=0}^{\infty}(3n+3)q^{3n^2+6nk+2n+4k}(q^{2n+2k+1}-1)+\sum_{n=0}^{\infty}q^{3n^2+6nk+4n+6k+1}
\]
and
\[
B:=\sum_{n=0}^{\infty}(3k-2)q^{3n^2+6nk+2n+4k}(q^{2n+2k+1}-1).
\]
We first transform $\sum_{i=0}^{\infty}\sum_{n=-k}^{k-1}g_i(n)q^i$ as the sum of two $q$-series involving $A$ and $B$ in the following lemma. Then we prove that each summand has non-positive coefficients, which yields Lemma \ref{lem3n+1-6}.

\begin{lem}\label{lem3n+1-1}
	For $k\ge 1$, we have
	\begin{equation}
	\sum_{i=0}^{\infty}\sum_{n=-k}^{k-1}g_i(n)q^i=1+\frac{q^{3k^2-2k}}{(q;q^2)_{\infty}^2(q^2;q^2)_{\infty}(q^4;q^4)_{\infty}^2}(A+B).
	\end{equation}
\end{lem}
\begin{proof}
		Using \eqref{quituple3n+1} and \eqref{equ-sum-i=0}, we have
	\begin{small}
		\begin{align}
		&\sum_{i=0}^{\infty}\sum_{n=-k}^{k-1}g_i(n)q^i\nonumber\\	=&\frac{1}{(q;q^2)_{\infty}^2(q^2;q^2)_{\infty}(q^4;q^4)_{\infty}^2}\sum_{n=-k}^{k-1}(3n+1)q^{3n^2+2n}\notag\\
			=&\frac{1}{(q;q^2)_{\infty}^2(q^2;q^2)_{\infty}(q^4;q^4)_{\infty}^2}\left(\sum_{n=-\infty}^{\infty}(3n+1)q^{3n^2+2n}-\sum_{n=-\infty}^{-k-1}(3n+1)q^{3n^2+2n}-\sum_{n=k}^{\infty}(3n+1)q^{3n^2+2n}\right)\notag\\
			=&1-\frac{1}{(q;q^2)_{\infty}^2(q^2;q^2)_{\infty}(q^4;q^4)_{\infty}^2}\left(\sum_{n=-\infty}^{-k-1}(3n+1)q^{3n^2+2n}+\sum_{n=k}^{\infty}(3n+1)q^{3n^2+2n}\right).\label{(3n+1)-to-k-1kto+}
		\end{align}
	\end{small}
	Changing the index $n$ in $\sum_{n=-\infty}^{-k-1}(3n+1)q^{3n^2+2n}$ into $-n-k-1$, we deduce that
	\begin{equation}\label{eq-3n+1-3n22n}
		\sum_{n=-\infty}^{-k-1}(3n+1)q^{3n^2+2n}=-\sum_{n=0}^{\infty}(3n+3k+2)q^{3n^2+3k^2+6nk+4n+4k+1}.
	\end{equation}
	Similarly, we change the index $n$ in $\sum_{n=k}^{\infty}(3n+1)q^{3n^2+2n}$ into $n+k$ to obtain
	\begin{equation}\label{eq-3n+1-3n22n-1}
		\sum_{n=k}^{\infty}(3n+1)q^{3n^2+2n}=\sum_{n=0}^{\infty}(3n+3k+1)q^{3n^2+6nk+3k^2+2n+2k}.
	\end{equation}
	Substituting \eqref{eq-3n+1-3n22n} and \eqref{eq-3n+1-3n22n-1} into \eqref{(3n+1)-to-k-1kto+}, we have
	\begin{align}
		&\sum_{i=0}^{\infty}\sum_{n=-k}^{k-1}g_i(n)q^i\notag\\
		=&1+\frac{1}{(q;q^2)_{\infty}^2(q^2;q^2)_{\infty}(q^4;q^4)_{\infty}^2}\left(\sum_{n=0}^{\infty}(3n+3k+2)q^{3n^2+3k^2+6nk+4n+4k+1}\right.\nonumber\\
		&\left.-\sum_{n=0}^{\infty}(3n+3k+1)q^{3n^2+6nk+3k^2+2n+2k}\right)\notag\\
		=&1+\frac{q^{3k^2-2k}}{(q;q^2)_{\infty}^2(q^2;q^2)_{\infty}(q^4;q^4)_{\infty}^2}\left(\sum_{n=0}^{\infty}(3n+3)q^{3n^2+6nk+2n+4k}(q^{2n+2k+1}-1)\right.\nonumber\\
		&\left.+\sum_{n=0}^{\infty}q^{3n^2+6nk+4n+6k+1}+\sum_{n=0}^{\infty}(3k-2)q^{3n^2+6nk+2n+4k}(q^{2n+2k+1}-1)\right)\notag\\
		=&1+\frac{q^{3k^2-2k}}{(q;q^2)_{\infty}^2(q^2;q^2)_{\infty}(q^4;q^4)_{\infty}^2}\left(A+B\right).
	\end{align}
\end{proof}

We next show that
\[\frac{A}{(q;q^2)_{\infty}^2(q^2;q^2)_{\infty}(q^4;q^4)_{\infty}^2}\]
has non-positive coefficients. To this end, we first rewrite $A$ as $-A_1-A_2-A_3$ in the following lemma, where
\begin{align*}
	A_1:=&\sum_{n=1}^{\infty}\frac{(-1)^{n+1}q^{n^2+(4k+1)n+6k}}{(-q^{2k+1};q^2)_{n+1}}\sum_{i=0}^{n}\frac{q^{2i+1}}{1+q^{2i+2k+1}},\\
	A_2:=&\sum_{n=2}^{\infty}\frac{(-1)^n(2n+3)q^{n^2+(4k+1)n+4k}}{(-q^{2k+1};q^2)_{n+1}},	\\
	A_3:=&\frac{-q^{6k+1}}{(1+q^{2k+1})^2}+\frac{3q^{4k}}{1+q^{2k+1}}-\frac{5q^{8k+2}}{(1+q^{2k+1})(1+q^{2k+3})}.	
\end{align*}
\begin{lem}\label{lem3n+1-2} For $k\ge 1$,
	\begin{equation}
	A=-(A_1+A_2+A_3).
	\end{equation}

\end{lem}

\begin{proof}
	Recall that Ramanujan found the following identity (see \cite[Entry 9.5.1]{andrews2005ramanujan})
	\begin{equation}\label{raidentity3n+1}
		\sum_{n=0}^{\infty}\frac{(-1)^na^{2n}q^{n(n+1)}}{(-aq;q^2)_{n+1}}=\sum_{n=0}^{\infty}a^{3n}q^{3n^2+2n}(1-aq^{2n+1}).
	\end{equation}
	We multiply $a^3$ and then differentiate with respect to $a$ on both sides of \eqref{raidentity3n+1} to get
	\begin{align}
		&\sum_{n=0}^{\infty}\frac{(-1)^{n+1}a^{2n+3}q^{n^2+n}}{(-aq;q^2)_{n+1}}\sum_{i=0}^{n}\frac{q^{2i+1}}{1+aq^{2i+1}}+\sum_{n=0}^{\infty}\frac{(-1)^n(2n+3)a^{2n+2}q^{n^2+n}}{(-aq;q^2)_{n+1}}\nonumber\\
		=&\sum_{n=0}^{\infty}(3n+3)a^{3n+2}q^{3n^2+2n}(1-aq^{2n+1})-\sum_{n=0}^{\infty}a^{3n+3}q^{3n^2+4n+1}\label{3n+1qiudaoweitihuan}.
	\end{align}
	Replacing $a$ by $q^{2k}$ in \eqref{3n+1qiudaoweitihuan} , we have
	\begin{align}
	-A=	&\sum_{n=0}^{\infty}(3n+3)q^{6nk+4k+3n^2+2n}(1-q^{2k+2n+1})-\sum_{n=0}^{\infty}q^{6nk+6k+3n^2+4n+1}\nonumber\\
		=&\sum_{n=0}^{\infty}\frac{(-1)^{n+1}q^{2k(2n+3)+n^2+n}}{(-q^{2k+1};q^2)_{n+1}}\sum_{i=0}^{n}\frac{q^{2i+1}}{1+q^{2k+2i+1}}+\sum_{n=0}^{\infty}\frac{(-1)^n(2n+3)q^{2k(2n+2)+n^2+n}}{(-q^{2k+1};q^2)_{n+1}}\nonumber\\
		=&\sum_{n=1}^{\infty}\frac{(-1)^{n+1}q^{n^2+(4k+1)n+6k}}{(-q^{2k+1};q^2)_{n+1}}\sum_{i=0}^{n}\frac{q^{2i+1}}{1+q^{2i+2k+1}}+\sum_{n=2}^{\infty}\frac{(-1)^n(2n+3)q^{n^2+(4k+1)n+4k}}{(-q^{2k+1};q^2)_{n+1}}\nonumber\\
		&+\left(\frac{-q^{6k+1}}{(1+q^{2k+1})^2}+\frac{3q^{4k}}{1+q^{2k+1}}-\frac{5q^{8k+2}}{(1+q^{2k+1})(1+q^{2k+3})}\right)\nonumber\\
		=&A_1+A_2+A_3.	
	\end{align}
\end{proof}
	
	The following three lemmas show that
	\[\frac{A_1}{(q;q^2)_{\infty}^2(q^2;q^2)_{\infty}(q^4;q^4)_{\infty}^2},\quad\frac{A_2}{(q;q^2)_{\infty}^2(q^2;q^2)_{\infty}(q^4;q^4)_{\infty}^2},\quad\frac{A_3}{(q;q^2)_{\infty}^2(q^2;q^2)_{\infty}(q^4;q^4)_{\infty}^2}\]
	all have non-negative coefficients, respectively. Thus combining Lemma \ref{lem3n+1-2}, we arrive at
	\[\frac{A}{(q;q^2)_{\infty}^2(q^2;q^2)_{\infty}(q^4;q^4)_{\infty}^2}\]
	has non-positive coefficients.

\begin{lem}\label{lem3n+1-3}
	\begin{equation}
	\frac{A_1}{(q;q^2)_{\infty}^2(q^2;q^2)_{\infty}(q^4;q^4)_{\infty}^2}
	\end{equation}
	has non-negative coefficients.
\end{lem}
\begin{proof}
	 {When $n=2j+1$} for some $j\ge 0$, 
\[\frac{(-1)^{n+1}q^{n^2+(4k+1)n+6k}}{(-q^{2k+1};q^2)_{n+1}}\sum_{i=0}^{n}\frac{q^{2i+1}}{1+q^{2i+2k+1}}=\frac{q^{4j^2+6j+10k+8kj+2}}{(-q^{2k+1};q^2)_{2j+2}}\sum_{i=0}^{2j+1}\frac{q^{2i+1}}{1+q^{2i+2k+1}},\]
and {when $n=2j+2$} for some $j\ge 0$, 
\[\frac{(-1)^{n+1}q^{n^2+(4k+1)n+6k}}{(-q^{2k+1};q^2)_{n+1}}\sum_{i=0}^{n}\frac{q^{2i+1}}{1+q^{2i+2k+1}}=-\frac{q^{4j^2+10j+14k+8kj+6}}{(-q^{2k+1};q^2)_{2j+3}}\sum_{i=0}^{2j+2}\frac{q^{2i+1}}{1+q^{2i+2k+1}}.\]
From the above analysis, we see that
\begin{align}
	&\frac{A_1}{(q;q^2)_{\infty}^2(q^2;q^2)_{\infty}(q^4;q^4)_{\infty}^2}\nonumber\\
	=&\frac{1}{(q;q^2)_{\infty}^2(q^2;q^2)_{\infty}(q^4;q^4)_{\infty}^2}\sum_{n=1}^{\infty}\frac{(-1)^{n+1}q^{n^2+(4k+1)n+6k}}{(-q^{2k+1};q^2)_{n+1}}\sum_{i=0}^{n}\frac{q^{2i+1}}{1+q^{2i+2k+1}}\nonumber\\
	=&\frac{1}{(q;q^2)_{\infty}^2(q^2;q^2)_{\infty}(q^4;q^4)_{\infty}^2}\sum_{j=0}^{\infty}\left(\frac{q^{4j^2+6j+10k+8kj+2}}{(-q^{2k+1};q^2)_{2j+2}}\sum_{i=0}^{2j+1}\frac{q^{2i+1}}{1+q^{2i+2k+1}}\right.\nonumber\\
	&\left.-\frac{q^{4j^2+10j+14k+8kj+6}}{(-q^{2k+1};q^2)_{2j+3}}\sum_{i=0}^{2j+2}\frac{q^{2i+1}}{1+q^{2i+2k+1}}\right)\nonumber\\
	=&\frac{1}{(q;q^2)_{\infty}^2(q^2;q^2)_{\infty}(q^4;q^4)_{\infty}^2}\sum_{j=0}^{\infty}\frac{q^{4j^2+6j+10k+8kj+2}}{(-q^{2k+1};q^2)_{2j+3}}\left(\sum_{i=0}^{2j+1}\frac{q^{2i+1}}{1+q^{2i+2k+1}}\right.\nonumber\\
	&\left.+q^{2k+4j+5}\sum_{i=0}^{2j+1}\frac{q^{2i+1}}{1+q^{2i+2k+1}}-q^{4j+4k+4}\left(\sum_{i=0}^{2j+1}\frac{q^{2i+1}}{1+q^{2i+2k+1}}+\frac{q^{4j+5}}{1+q^{4j+2k+5}}\right)\right)\nonumber\\
	=&\frac{1}{(q;q^2)_{\infty}^2(q^2;q^2)_{\infty}(q^4;q^4)_{\infty}^2}\sum_{j=0}^{\infty}\frac{q^{4j^2+6j+10k+8kj+2}}{(-q^{2k+1};q^2)_{2j+3}}\left((1-q^{4j+4k+4})\sum_{i=0}^{2j+1}\frac{q^{2i+1}}{1+q^{2i+2k+1}}\right.\nonumber\\
	&\left.+q^{2k+4j+5}\sum_{i=1}^{2j+1}\frac{q^{2i+1}}{1+q^{2i+2k+1}}+\left(q^{2k+4j+5}\frac{q}{1+q^{2k+1}}-q^{4k+4j+4}\frac{q^{4j+5}}{1+q^{4j+2k+5}}\right)\right)\nonumber\\
	=&\frac{1}{(q^2;q^2)_{\infty}(q^4;q^4)^2_{\infty}}\sum_{j=0}^{\infty}\frac{q^{4j^2+6j+10k+8kj+2}}{(-q^{2k+1};q^2)_{2j+3}(q;q^2)_{\infty}}\left(\sum_{i=0}^{2j+1}\frac{q^{2i+1}(1-q^{4j+4k+4})}{(1+q^{2i+2k+1})(q;q^2)_{\infty}}\right.\nonumber\\
	&\left.+q^{2k+4j+5}\sum_{i=1}^{2j+1}\frac{q^{2i+1}}{(1+q^{2i+2k+1})(q;q^2)_{\infty}}+\frac{q^{2k+4j+6}(1-q^{4j+2k+3})+q^{8j+4k+11}(1-q^{2k-1})}{(q;q^2)_{\infty}(1+q^{2k+1})(1+q^{4j+2k+5})}\right)\label{a13n+1huajianfeifu}.
\end{align}
We proceed to show that \eqref{a13n+1huajianfeifu} has non-negative coefficients. First, it is easy to see that
\begin{equation}
	\frac{1}{(-q^{2k+1};q^2)_{2j+3}(q;q^2)_{\infty}}=\frac{(-q^{2k+4j+7};q^2)_{\infty}}{(q;q^2)_{\infty}(-q^{2k+1};q^2)_{\infty}}=\frac{(-q^{2k+4j+7};q^2)_{\infty}}{(q;q^2)_{k}(q^{4k+2};q^4)_{\infty}}\label{a13n+1part1}
\end{equation}
has non-negative coefficients. Moverover
\begin{align}
	\frac{1-q^{4j+4k+4}}{(1+q^{2i+2k+1})(q;q^2)_{\infty}}=\frac{1+q+q^2+\cdots+q^{4j+4k+3}}{(1-q^{4i+4k+2})(q^3;q^2)_{i+k-1}(q^{2i+2k+3};q^2)_{\infty}}\label{a13n+1part2}
\end{align}
has non-negative coefficients. Similarly,
\begin{equation}
	\frac{1}{(1+q^{2i+2k+1})(q;q^2)_{\infty}}=\frac{1}{(1-q^{4k+4i+2})(q;q^2)_{k+i}(q^{2k+2i+3};q^2)_\infty}\label{a13n+1part3}
\end{equation}
has non-negative coefficients. Furthermore,

\begin{align}
	&\frac{q^{2k+4j+6}(1-q^{4j+2k+3})+q^{8j+4k+11}(1-q^{2k-1})}{(q;q^2)_{\infty}(1+q^{2k+1})(1+q^{4j+2k+5})}\nonumber\\
	=&\frac{q^{2k+4j+6}}{(q;q^2)_k(q^{2k+3};q^2)_{2j}(q^{4j+2k+7};q^2)_{\infty}(1-q^{4k+2})(1-q^{8j+4k+10})}\nonumber\\
	&+\frac{q^{8j+4k+11}}{(q;q^2)_{k-1}(q^{2k+3};q^2)_{2j+1}(q^{4j+2k+7};q^2)_{\infty}(1-q^{4k+2})(1-q^{8j+4k+10})}\label{a13n+1part4}
\end{align}
also has non-negative coefficients.
Combining \eqref{a13n+1part1}, \eqref{a13n+1part2}, \eqref{a13n+1part3} and \eqref{a13n+1part4}, we complete the proof.
	
\end{proof}

\begin{lem}\label{lem3n+1-4}
	\begin{equation}
\frac{A_2}{(q;q^2)_{\infty}^2(q^2;q^2)_{\infty}(q^4;q^4)_{\infty}^2}
\end{equation}
has non-negative coefficients.
\end{lem}
\begin{proof}
	First, we have
	\begin{align}
		&\frac{A_2}{(q;q^2)_{\infty}^2(q^2;q^2)_{\infty}(q^4;q^4)_{\infty}^2}\nonumber\\
		=&\frac{1}{(q;q^2)_{\infty}^2(q^2;q^2)_{\infty}(q^4;q^4)_{\infty}^2}\sum_{n=2}^{\infty}\frac{(-1)^n(2n+3)q^{n^2+(4k+1)n+4k}}{(-q^{2k+1};q^2)_{n+1}}\nonumber\\
		=&\frac{1}{(q;q^2)_{\infty}^2(q^2;q^2)_{\infty}(q^4;q^4)_{\infty}^2}\sum_{j=1}^{\infty}\left(\frac{(4j+3)q^{4j^2+8kj+2j+4k}}{(-q^{2k+1};q^2)_{2j+1}}-\frac{(4j+5)q^{4j^2+8kj+6j+8k+2}}{(-q^{2k+1};q^2)_{2j+2}}\right)\nonumber\\
		=&\frac{1}{(q;q^2)_{\infty}^2(q^2;q^2)_{\infty}(q^4;q^4)_{\infty}^2}\sum_{j=1}^{\infty}\frac{q^{4j^2+8kj+2j+4k}}{(-q^{2k+1};q^2)_{2j+2}}\Bigg((4j+3)(1+q^{2k+4j+3})-q^{4j+4k+2}(4j+5)\Bigg)\nonumber\\
		=&\frac{1}{(q^2;q^2)_{\infty}(q^4;q^4)_{\infty}^2}\sum_{j=1}^{\infty}\frac{q^{4j^2+8kj+2j+4k}}{(q;q^2)_{\infty}(-q^{2k+1};q^2)_{2j+2}}\Bigg(\frac{(4j+3)(1-q^{4j+4k+2})}{(q;q^2)_{\infty}}\nonumber\\
		&+\frac{2q^{2k+4j+3}(1-q^{2k-1})+(4j+1)q^{2k+4j+3}}{(q;q^2)_{\infty}}\Bigg).\label{a23n+1huajianfeifu}
	\end{align}
On the one hand,
	\begin{align}
		\frac{1}{(q;q^2)_{\infty}(-q^{2k+1};q^2)_{2j+2}}=\frac{(-q^{2k+4j+5};q^2)_{\infty}}{(q;q^2)_{\infty}(-q^{2k+1};q^2)_{\infty}}
		=\frac{(-q^{2k+4j+5};q^2)_{\infty}}{(q;q^2)_{k}(q^{4k+2};q^2)_{\infty}}\label{a23n+1part1}
	\end{align}
	has non-negative coefficients. On the other hand, it is easy to check that
	$$\frac{1-q^{4j+4k+2}}{(q;q^2)_{\infty}}=\frac{1+q+\cdots+q^{4j+4k+1}}{(q^3;q^2)_{\infty}}$$
	and
	$$
	\frac{1-q^{2k-1}}{(q,q^2)_{\infty}}=\frac{1}{(q;q^2)_{k-1}(q^{2k+1},q^2)_{\infty}}
	$$
both have non-negative coefficients. From the above analysis, we complete the proof.
\end{proof}

\begin{lem}\label{lem3n+1-5}
	\begin{equation}
	\frac{A_3}{(q;q^2)_{\infty}^2(q^2;q^2)_{\infty}(q^4;q^4)_{\infty}^2}
	\end{equation}
	has non-negative coefficients.
\end{lem}	
\begin{proof}
	It is easy to see that
	\begin{align}
	&\frac{A_3}{(q;q^2)_{\infty}^2(q^2;q^2)_{\infty}(q^4;q^4)_{\infty}^2}\nonumber\\
		=&\frac{1}{(q;q^2)_{\infty}^2(q^2;q^2)_{\infty}(q^4;q^4)_{\infty}^2}\left(\frac{-q^{6k+1}}{(1+q^{2k+1})^2}+\frac{3q^{4k}}{1+q^{2k+1}}-\frac{5q^{8k+2}}{(1+q^{2k+1})(1+q^{2k+3})}\right)\nonumber\\[3pt]
		=&\frac{-q^{6k+1}(1+q^{2k+3})+3q^{4k}(1+q^{2k+1})(1+q^{2k+3})-5q^{8k+2}(1+q^{2k+1})}{(q;q^2)_{\infty}^2(q^2;q^2)_{\infty}(q^4;q^4)_{\infty}^2(1+q^{2k+1})^2(1+q^{2k+3})}\nonumber\\[3pt]
		=&\frac{(2q^{6k+1}+2q^{4k})(1-q^{4k+2})+(2q^{8k+4}+3q^{6k+3})(1-q^{2k-1})+q^{4k}(1-q^{6k+3})}{(q;q^2)_{\infty}^2(q^2;q^2)_{\infty}(q^4;q^4)_{\infty}^2(1+q^{2k+1})^2(1+q^{2k+3})}\nonumber\\
		=&\frac{1}{(q^4;q^4)_{\infty}^2(1-q^{4k+2})^2(1-q^{4k+6})}
		\left(\frac{2q^{6k+1}+2q^{4k}}{(q;q^2)_{k}^2(q^{2k+3};q^2)_{\infty}(q^{2k+5};q^2)_{\infty}(q^2;q^2)_{2k}(q^{4k+4};q^2)_{\infty}}\right.\nonumber\\
		&+\left.\frac{2q^{8k+4}+3q^{6k+3}}{(q;q^2)_{k-1}^2(1-q^{2k-1})(q^{2k+3};q^2)_{\infty}(q^{2k+5};q^2)_{\infty}(q^2;q^2)_{\infty}}\right.\nonumber\\
		&+\left.\frac{q^{4k}}{(q;q^2)_{k}^{2}(q^2;q^2)_{\infty}(q^{2k+5};q^2)_{2k-1}^2(1-q^{2k+3})(1-q^{6k+3})(q^{6k+5};q^2)_{\infty}^2}\right).\label{a33n+1huajianhou}
	\end{align}
	Clearly, each summand in \eqref{a33n+1huajianhou} has non-negative coefficients. Thus we complete the proof.
\end{proof}

We are now in a position to prove Lemma \ref{lem3n+1-6}.

\noindent{\it Proof of Lemma \ref{lem3n+1-6}.}
	From Lemma \ref{lem3n+1-1}, it suffices to show that both $\dfrac{A}{(q;q^2)_{\infty}^2(q^2;q^2)_{\infty}(q^4;q^4)_{\infty}^2}$
	 and
	 $\dfrac{B}{(q;q^2)_{\infty}^2(q^2;q^2)_{\infty}(q^4;q^4)_{\infty}^2}$
	 has non-positive coefficients. By Lemma \ref{lem3n+1-2} $\sim$ Lemma \ref{lem3n+1-5}, we see that
	 $$\frac{A}{(q;q^2)_{\infty}^2(q^2;q^2)_{\infty}(q^4;q^4)_{\infty}^2}$$  has non-positive coefficients. We proceed to show that $$\frac{B}{(q;q^2)_{\infty}^2(q^2;q^2)_{\infty}(q^4;q^4)_{\infty}^2}$$
	 also has non-positive coefficients.
	
In fact,
	\begin{align}
		&\frac{B}{(q;q^2)_{\infty}^2(q^2;q^2)_{\infty}(q^4;q^4)_{\infty}^2}\nonumber\\
		=&\frac{1}{(q;q^2)_{\infty}^2(q^2;q^2)_{\infty}(q^4;q^4)_{\infty}^2}\sum_{n=0}^{\infty}(3k-2)q^{3n^2+6nk+2n+4k}(q^{2n+2k+1}-1)\nonumber\\
		=&-\sum_{n=0}^{\infty}\frac{(3k-2)q^{3n^2+6nk+2n+4k}}{(q;q^2)_{\infty}(q^2;q^2)_{\infty}(q^4;q^4)_{\infty}^2(q;q^2)_{n+k}(q^{2n+2k+3};q^2)_\infty},\nonumber
	\end{align}
which	has non-positive coefficients.	 \qed

Now we prove  Theorem \ref{thm-3n+1}.

{\noindent \it Proof of Theorem \ref{thm-3n+1}.}
By the definition of $g_i(n)$, we have  $g_i(n)\ge 0$ when $n\ge 0$ and  $g_i(n)\le 0$ when $n <0$. Moreover, from Theorem \ref{chan2016truncated-1} we know that $\sum_{n=-k}^{k}g_i(n)\ge 0$ for any $i\ge 0$ and $k\ge 1$. By Lemma \ref{lem3n+1-6}, we see that $\sum_{n=-k}^{k-1}g_i(n)\le 0$ for any $i\ge 1$ and $k\ge 1$. Hence, using Lemma \ref{lemqn}, we deduce that $\Sg(a+b)\sum_{n=a}^{b}g_i(n)\ge 0$ for any $a,b\in \mathbb{Z}$ and $i\ge 1$, which means that the coefficients of $q^i$ in \eqref{3n+1atobmainthm} are non-negative for any $i\ge 1$. This completes the proof.
\qed

\section{Proof of Theorem \ref{thm-3}}\label{sectionthm3}
This section is devoted to proving Theorem \ref{thm-3}. 
 For any $n\in\mathbb{Z}$ and $i\ge 0$, define
\[\sum_{i=0}^\infty h_i(n)q^i=\frac{(6n+1)q^{3n^2+n}}{(q^2;q^2)_{\infty}^{3}(q^2;q^4)_{\infty}^{2}}.\]

Similar with the proof of Theorem \ref{thm-3n+1}, we first rewrite $\sum_{i=0}^{\infty}\sum_{n=-k}^{k-1}h_i(n)q^i$ as the sum of two $q$-series in Lemma \ref{lem6n+1-1}. With the same arguments as in the proof of Theorem \ref{thm-3n+1}, we deduce that $\sum_{n=-k}^{k-1}h_i(n)\le 0$ in Lemma \ref{lem6n+1-5}. Together with Theorem \ref{thmchan2016truncated6n+1} and Lemma \ref{lemqn}, we will finish the proof of Theorem \ref{thm-3}.

We first show that $h_i(n)$ satisfies the first inequality in \eqref{equ-sum-n-k} as stated below. 
\begin{lem}\label{lem6n+1-5}
	For $k \geq 1$ and $i\ge 1$,
	\begin{equation}\label{lem5eq6n+1-ktok-1}
		\sum_{n=-k}^{k-1}h_i(n)\le 0.
	\end{equation}
\end{lem}

Define
\[
C:=\sum_{n=0}^{\infty}(6n+4)q^{3n^2+(6k+1)n+2k}(q^{4n+4k+2}-1)+4\sum_{n=0}^{\infty}q^{3n^2+(6k+1)n+4n+6k+2}
\]
and
\[
D:=\sum_{n=0}^{\infty}(6k-3)q^{3n^2+(6k+1)n+2k}(q^{4n+4k+2}-1).
\]
To prove Lemma \ref{lem6n+1-5}, we first transform $\sum_{i=0}^{\infty}\sum_{n=-k}^{k-1}h_i(n)q^i$ as in the following lemma.

\begin{lem}\label{lem6n+1-1}
	For $k \ge 1$, we have
	\begin{equation}
	\sum_{i=0}^{\infty}\sum_{n=-k}^{k-1}h_i(n)q^i=1+\frac{1}{(q^2;q^2)_{\infty}^{3}(q^2;q^4)_{\infty}^{2}}(C+D).
	\end{equation}
\end{lem}
\begin{proof}
	Using \eqref{quituple6n+1},	we have
	\begin{small}
		\begin{align}\label{eq-quint-1.11-1}
			&\sum_{i=0}^{\infty}\sum_{n=-k}^{k-1}h_i(n)q^i\nonumber\\
			=&\frac{1}{(q^2;q^2)_{\infty}^{3}(q^2;q^4)_{\infty}^{2}}\sum_{n=-k}^{k-1}(6n+1)q^{3n^2+n}\notag\\
			=&\frac{1}{(q^2;q^2)_{\infty}^{3}(q^2;q^4)_{\infty}^{2}}\left(\sum_{n=-\infty}^{\infty}(6n+1)q^{3n^2+n}-\sum_{n=-\infty}^{-k-1}(6n+1)q^{3n^2+n}-\sum_{n=k}^{\infty}(6n+1)q^{3n^2+n}\right)\notag\\
			=&1-\frac{1}{(q^2;q^2)_{\infty}^{3}(q^2;q^4)_{\infty}^{2}}\left(\sum_{n=-\infty}^{-k-1}(6n+1)q^{3n^2+n}+\sum_{n=k}^{\infty}(6n+1)q^{3n^2+n}\right).
		\end{align}
	\end{small}
	Changing the index $n$ in $\sum_{n=-\infty}^{-k-1}(6n+1)q^{3n^2+n}$	into $-n-k-1$, we deduce that
	\begin{equation}\notag
		\sum_{n=-\infty}^{-k-1}(6n+1)q^{3n^2+n}=-\sum_{n=0}^{\infty}(6(n+k)+5)q^{3n^2+3k^2+6nk+5n+5k+2}.
	\end{equation}
	Similarly, replacing the index  $n$ in $\sum_{n=k}^{\infty}(6n+1)q^{3n^2+n}$ into $n+k$, we have	
	\begin{equation}\notag
		\sum_{n=k}^{\infty}(6n+1)q^{3n^2+n}=\sum_{n=0}^{\infty}(6(n+k)+1)q^{3n^2+6nk+3k^2+n+k}.
	\end{equation}
	Plugging into \eqref{eq-quint-1.11-1}, we have	
	\begin{align*}
		&\sum_{i=0}^{\infty}\sum_{n=-k}^{k-1}h_i(n)q^i\\
		=&1+\frac{1}{(q^2;q^2)_{\infty}^{3}(q^2;q^4)_{\infty}^{2}}\left(\sum_{n=0}^{\infty}(6(n+k)+5)q^{3n^2+3k^2+6nk+5n+5k+2}\right.\\
		&\left.-\sum_{n=0}^{\infty}(6(n+k)+1)q^{3n^2+6nk+3k^2+n+k}\right)\\
		=&1+\frac{q^{3k^2-k}}{(q^2;q^2)_{\infty}^{3}(q^2;q^4)_{\infty}^{2}}\left(\sum_{n=0}^{\infty}(6n+4)q^{3n^2+(6k+1)n+2k}(q^{4n+4k+2}-1)\right.\\
		&\left.+4\sum_{n=0}^{\infty}q^{3n^2+(6k+1)n+4n+6k+2}+\sum_{n=0}^{\infty}(6k-3)q^{3n^2+(6k+1)n+2k}(q^{4n+4k+2}-1)\right)\\
		=&1+\frac{q^{3k^2-k}}{(q^2;q^2)_{\infty}^{3}(q^2;q^4)_{\infty}^{2}}\left(C+D\right).
	\end{align*}
\end{proof}
We now prove that $C/(q^2;q^2)_{\infty}^{3}(q^2;q^4)_{\infty}^{2}$ has non-positive coefficients. To this end, we  rewrite $C$ as $-2C_1-2C_2-2C_3-2C_4$ in the following lemma, 	where
\[
\begin{aligned}
	C_{1}:=&\sum_{j=0}^{\infty}\frac{q^{8kj+8k+4j^2+6j+2}}{(-q^{2k+2};q^2)_{2j+2}}\left(1-q^{4k+4j+4}\right)\sum_{i=1}^{2j+1}\frac{q^{2i}}{1+q^{2i+2k}},\\
	C_{2}:=&\sum_{j=0}^{\infty}\frac{q^{8kj+8k+4j^2+6j+2}}{(-q^{2k+2};q^2)_{2j+2}}q^{2k+4j+4}\sum_{i=2}^{2j+1}\frac{q^{2i}}{1+q^{2i+2k}},\\
	C_{3}:=&\sum_{j=0}^{\infty}\frac{q^{8kj+8k+4j^2+6j+2}}{(-q^{2k+2};q^2)_{2j+2}}q^{2k+4j+4}\left(\frac{q^{2}}{1+q^{2k+2}}-\frac{q^{2k+4j+4}}{1+q^{2k+4j+4}}\right),\\
	C_{4}:=&\sum_{n=0}^{\infty}\frac{(-1)^n(2n+2)q^{4kn+2k+n^2+n}}{(-q^{2k+2};q^2)_n}.
\end{aligned}
\]

\begin{lem}\label{lem6n+1-2}
	For $k\ge 1$, 	
	\begin{equation}
		-C=2(C_1+C_2+C_3+C_4).
	\end{equation}
\end{lem}
\begin{proof}
	Recall the following identity due to \cite[Entry 9.4.1]{andrews2005ramanujan},
	\[
	\sum_{n=0}^{\infty}\frac{(-1)^na^{2n}q^{n(n+1)/2}}{(-aq;q)_n}=\sum_{n=0}^{\infty}a^{3n}q^{n(3n+1)/2}(1-a^2q^{2n+1}).
	\]
	Replacing $q$ by $q^2$ and multiply $a^2$ in both sides of  the above equation, we  obtain
	\[
	\sum_{n=0}^{\infty}\frac{(-1)^na^{2n+2}q^{n(n+1)}}{(-aq^2;q^2)_n}=\sum_{n=0}^{\infty}a^{3n+2}q^{n(3n+1)}(1-a^2q^{4n+2}).
	\]
	Differentiate both sides with respect to $a$ to get
	\begin{align}
		&\sum_{n=1}^{\infty}\frac{(-1)^{n+1}a^{2n+2}q^{n^2+n}}{(-aq^2;q^2)_{n}}\sum_{i=1}^{n}\frac{q^{2i}}{1+aq^{2i}}+\sum_{n=0}^{\infty}\frac{(-1)^n(2n+2)a^{2n+1}q^{n^2+n}}{(-aq^2;q^2)_{n}}\nonumber\\
		=&\sum_{n=0}^{\infty}(3n+2)a^{3n+1}q^{3n^2+n}(1-a^2q^{4n+2})-2\sum_{n=0}^{\infty}a^{3n+3}q^{3n^2+5n+2}.\label{qiudaohou}
	\end{align}
	Replacing $a$ by $q^{2k}$ in \eqref{qiudaohou} and multiplying both sides by $2$, we get
	\begin{align}
		-C=&\sum_{n=0}^{\infty}(6n+4)q^{6kn+2k+3n^2+n}(1-q^{4k+4n+2})-4\sum_{n=0}^{\infty}q^{6kn+6k+3n^2+5n+2}\nonumber\\
		=&2\sum_{n=1}^{\infty}\frac{(-1)^{n+1}q^{4kn+4k+n^2+n}}{(-q^{2k+2};q^2)_n}\sum_{i=1}^{n}\frac{q^{2i}}{1+q^{2i+2k}}+2\sum_{n=0}^{\infty}\frac{(-1)^n(2n+2)q^{4kn+2k+n^2+n}}{(-q^{2k+2};q^2)_n}.\label{eqC1}
	\end{align}
	Noticed that the second summand in \eqref{eqC1} is $2C_4$ and the first summand in \eqref{eqC1} can be further transformed as follows.
	\begin{align}
		&2\sum_{n=1}^{\infty}\frac{(-1)^{n+1}q^{4kn+4k+n^2+n}}{(-q^{2k+2};q^2)_n}\sum_{i=1}^{n}\frac{q^{2i}}{1+q^{2i+2k}}\nonumber\\
		=&2\sum_{j=0}^{\infty}\left(\frac{q^{4k(2j+1)+4k+(2j+1)^2+(2j+1)}}{(-q^{2k+2};q^2)_{2j+1}}\sum_{i=1}^{2j+1}\frac{q^{2i}}{1+q^{2i+2k}}\right.\nonumber\\
		&\left.-\frac{q^{4k(2j+2)+4k+(2j+2)^2+(2j+2)}}{(-q^{2k+2};q^2)_{2j+2}}\sum_{i=1}^{2j+2}\frac{q^{2i}}{1+q^{2i+2k}}\right)\nonumber\\
		=&2\sum_{j=0}^{\infty}\frac{q^{8kj+8k+4j^2+6j+2}}{(-q^{2k+2};q^2)_{2j+2}}\left(\left(1+q^{2k+4j+4}\right)\sum_{i=1}^{2j+1}\frac{q^{2i}}{1+q^{2i+2k}}\right.\nonumber\\
		&\left.-q^{4k+4j+4}\sum_{i=1}^{2j+1}\frac{q^{2i}}{1+q^{2i+2k}}-\frac{q^{8j+4k+8}}{1+q^{4j+2k+4}}\right)\nonumber\\
		=&2\sum_{j=0}^{\infty}\frac{q^{8kj+8k+4j^2+6j+2}}{(-q^{2k+2};q^2)_{2j+2}}\left(\left(1-q^{4k+4j+4}\right)\sum_{i=1}^{2j+1}\frac{q^{2i}}{1+q^{2i+2k}}\right.\nonumber\\
		&\left.+q^{2k+4j+4}\sum_{i=1}^{2j+1}\frac{q^{2i}}{1+q^{2i+2k}}-\frac{q^{8j+4k+8}}{1+q^{4j+2k+4}}\right)\nonumber\\	=&2\sum_{j=0}^{\infty}\frac{q^{8kj+8k+4j^2+6j+2}}{(-q^{2k+2};q^2)_{2j+2}}\left((1-q^{4k+4j+4})\sum_{i=1}^{2j+1}\frac{q^{2i}}{1+q^{2i+2k}}\right.\nonumber\\
		&\left.+q^{2k+4j+4}\sum_{i=2}^{2j+1}\frac{q^{2i}}{1+q^{2i+2k}}	+q^{2k+4j+4}\left(\frac{q^{2}}{1+q^{2k+2}}-\frac{q^{2k+4j+4}}{1+q^{2k+4j+4}}\right)\right)\nonumber\\
		=&2(C_1+C_2+C_3).\label{c1c2c3}
		\end{align}
		Combining \eqref{eqC1} and \eqref{c1c2c3}, we deduce that $-C=2(C_1+C_2+C_3+C_4)$.
\end{proof}
Next we prove that
\[
\frac{C_1}{(q^2;q^2)_{\infty}^{3}(q^2;q^4)_{\infty}^{2}},\quad
\frac{C_2}{(q^2;q^2)_{\infty}^{3}(q^2;q^4)_{\infty}^{2}},\quad
\frac{C_3}{(q^2;q^2)_{\infty}^{3}(q^2;q^4)_{\infty}^{2}},\quad
\frac{C_4}{(q^2;q^2)_{\infty}^{3}(q^2;q^4)_{\infty}^{2}}\]
all have non-negative coefficients in the following lemma.
Together with Lemma \ref{lem6n+1-2}, we will deduce the non-positivity of $\dfrac{C}{(q^2;q^2)_{\infty}^{3}(q^2;q^4)_{\infty}^{2}}$.

\begin{lem}\label{lem6n+1-c1}
For $k\ge 1$,\[
\frac{C_1}{(q^2;q^2)_{\infty}^{3}(q^2;q^4)_{\infty}^{2}},\quad
\frac{C_2}{(q^2;q^2)_{\infty}^{3}(q^2;q^4)_{\infty}^{2}},\quad
\frac{C_3}{(q^2;q^2)_{\infty}^{3}(q^2;q^4)_{\infty}^{2}},\quad
\frac{C_4}{(q^2;q^2)_{\infty}^{3}(q^2;q^4)_{\infty}^{2}}\]
all	have non-negative coefficients.
\end{lem}
\begin{proof}
	First,
	\begin{align}\label{a11-0} &\frac{C_{1}}{(q^2;q^2)_{\infty}^{3}(q^2;q^4)_{\infty}^{2}}\notag\\ =&\frac{1}{(q^2;q^2)_{\infty}^{3}(q^2;q^4)_{\infty}^{2}}\sum_{j=0}^{\infty}\frac{q^{8kj+8k+4j^2+6j+2}}{(-q^{2k+2};q^2)_{2j+2}}\left(1-q^{4k+4j+4}\right)\sum_{i=1}^{2j+1}\frac{q^{2i}}{1+q^{2i+2k}}\nonumber\\		=&\frac{1}{(q^2;q^4)_{\infty}^{2}}\sum_{j=0}^{\infty}\Big(\frac{1}{(q^2;q^2)_{\infty}(-q^{2k+2};q^2)_{2j+2}}\cdot \frac{1-q^{4k+4j+4}}{(q^2;q^2)_{\infty}}\cdot\notag\\ &\sum_{i=1}^{2j+1}\frac{q^{2i}}{(1+q^{2i+2k})(q^2;q^2)_{\infty}}\cdot {q^{8kj+8k+4j^2+6j+2}}\Big).
	\end{align}
	Note that
	\begin{align}\label{a11-1}
		\frac{1}{(q^2;q^2)_{\infty}(-q^{2k+2};q^2)_{2j+2}}=\frac{1}{(q^2;q^2)_{\infty}}\frac{(-q^{2k+4j+6};q^2)_{\infty}}{(-q^{2k+2};q^2)_{\infty}}=\frac{1}{(q^2;q^2)_{k}}\frac{(-q^{2k+4j+6};q^2)_{\infty}}{(q^{4k+4};q^4)_{\infty}}
	\end{align}
	has non-negative coefficients. Moreover,
	\begin{equation}\label{a11-2}\frac{(1-q^{4k+4j+4})}{(q^2;q^2)_{\infty}}=\frac{1+q^2+\cdots+q^{4k+4j+2}}{(q^4;q^2)_\infty}\end{equation}
	also has non-negative coefficients. Furthermore,
	\begin{equation}\label{a11-3}
		\frac{q^{2i}}{(1+q^{2i+2k})(q^2;q^2)_{\infty}}=\frac{q^{2i}}{(1-q^{4i+4k})(q^2;q^2)_{i+k-1}(q^{2i+2k+2};q^2)_\infty}.
	\end{equation}
	Substituting \eqref{a11-1}, \eqref{a11-2} and \eqref{a11-3} into \eqref{a11-0}, we find that $C_{1}/{(q^2;q^2)_{\infty}^{3}(q^2;q^4)_{\infty}^{2}}$ has non-negative coefficients.
	
Second,	notice that
\begin{align}
	&\frac{C_2}{(q^2;q^2)_{\infty}^{3}(q^2;q^4)_{\infty}^{2}}\nonumber\\
	=&\frac{1}{(q^2;q^2)_{\infty}^{3}(q^2;q^4)_{\infty}^{2}}\sum_{j=0}^{\infty}\frac{q^{8kj+8k+4j^2+6j+2}}{(-q^{2k+2};q^2)_{2j+2}}q^{2k+4j+4}\sum_{i=2}^{2j+1}\frac{q^{2i}}{1+q^{2i+2k}}\nonumber\\
	=&\frac{1}{(q^2;q^2)_{\infty}(q^2;q^4)_{\infty}^{2}}\sum_{j=0}^{\infty}\frac{q^{8kj+10k+4j^2+10j+6}}{(q^2;q^2)_{\infty}(-q^{2k+2};q^2)_{2j+2}}\sum_{i=2}^{2j+1}\frac{q^{2i}}{(q^2;q^2)_{\infty}(1+q^{2i+2k})}.\label{eq6n+1c2}
	\end{align}
Using \eqref{a11-1} and \eqref{a11-3}, we see that \eqref{eq6n+1c2} has non-negative coefficients. 
	
Third,
		\begin{equation}\notag
			\begin{aligned}
				C_{3}=&\sum_{j=0}^{\infty}\frac{q^{8kj+10k+4j^2+10j+6}}{(-q^{2k+2};q^2)_{2j+2}}\left(\frac{q^2}{1+q^{2k+2}}-\frac{q^{2k+4j+4}}{1+q^{2k+4j+4}}\right)\\
				=&\sum_{j=0}^{\infty}\frac{q^{8kj+10k+4j^2+10j+6}}{(-q^{2k+2};q^2)_{2j+2}}\cdot\frac{q^2+q^{2k+4j+6}-q^{2k+4j+4}-q^{4k+4j+6}}{(1+q^{2k+2})(1+q^{2k+4j+4})}\\
				=&\sum_{j=0}^{\infty}\frac{q^{8kj+10k+4j^2+10j+8}}{(-q^{2k+2};q^2)_{2j+2}}\cdot\frac{\left(1-q^{2k+4j+2}\right)+q^{2k+4j+4}\left(1-q^{2k}\right) } {(1+q^{2k+2})(1+q^{2k+4j+4})},
			\end{aligned}
		\end{equation}
		thus
		\begin{align}\label{a13-0}
			&\frac{C_{3}}{(q^2;q^2)_{\infty}^{3}(q^2;q^4)_{\infty}^{2}}\nonumber\\
			=&\frac{1}{(q^2;q^2)_{\infty}^{3}(q^2;q^4)_{\infty}^{2}}\sum_{j=0}^{\infty}\frac{q^{8kj+10k+4j^2+10j+8}}{(-q^{2k+2};q^2)_{2j+2}}\cdot\frac{\left(1-q^{2k+4j+2}\right)+q^{2k+4j+4}\left(1-q^{2k}\right) } {(1+q^{2k+2})(1+q^{2k+4j+4})}\nonumber\\
			=&\frac{1}{(q^2;q^4)_{\infty}^{2}}\sum_{j=0}^{\infty}\frac{q^{8kj+10k+4j^2+10j+8}}{(q^2;q^2)_\infty(-q^{2k+2};q^2)_{2j+2}}\cdot\frac{\left(1-q^{2k+4j+2}\right)+q^{2k+4j+4}\left(1-q^{2k}\right) } {(q^2;q^2)_\infty^2(1+q^{2k+2})(1+q^{2k+4j+4})}.
		\end{align}
		Note that
		\begin{align}\label{a13-1}
			&\frac{\left(1-q^{2k+4j+2}\right)+q^{2k+4j+4}\left(1-q^{2k}\right)} {(q^2;q^2)_{\infty}^{2}(1+q^{2k+2})(1+q^{2k+4j+4})}\nonumber\\[3pt]
			=&\frac{\left(1-q^{2k+4j+2}\right)+q^{2k+4j+4}\left(1-q^{2k}\right)}{(q^2;q^2)_\infty}\cdot\frac{1}{(q^2;q^2)_\infty (1+q^{2k+2})(1+q^{2k+4j+4})}\nonumber\\[3pt]
			=&\left(\frac{1}{(q^2;q^2)_{k+2j}(q^{2k+4j+4};q^2)_\infty}+\frac{q^{2k+4j+4}}{(q^2;q^2)_{k-1}(q^{2k+2};q^2)_\infty}\right)\cdot\nonumber\\[3pt]
			&\frac{1}{(q^2;q^2)_k(1-q^{4k+4})(q^{2k+4};q^2)_{2j}(1-q^{4k+8j+8})(q^{2k+4j+6};q^2)_\infty},
		\end{align}
		which has non-negative coefficients. Combining \eqref{a11-1}, \eqref{a13-0} and \eqref{a13-1}, we derive that  $C_{3}/{(q^2;q^2)_{\infty}^{3}(q^2;q^4)_{\infty}^{2}}$ has non-negative coefficients.  

Finally, it is easy to see that
\begin{align}
	C_4=&\sum_{n=0}^{\infty}\frac{(-1)^{n}(2n+2)q^{4kn+2k+n^2+n}}{(-q^{2k+2};q^2)_{n}}\nonumber\\
	=&\sum_{j=0}^{\infty}\left(\frac{(4j+2)q^{8kj+2k+4j^2+2j}}{(-q^{2k+2};q^2)_{2j}}-\frac{(4j+4)q^{8kj+6k+4j^2+6j+2}}{(-q^{2k+2};q^2)_{2j+1}}\right)\nonumber\\
	=&\sum_{j=0}^{\infty}\frac{q^{8kj+2k+4j^2+2j}}{(-q^{2k+2};q^2)_{2j+1}}\left((4j+2)(1+q^{2k+4j+2})-(4j+4)q^{4k+4j+2}\right)\nonumber\\
	=&\sum_{j=0}^{\infty}\frac{q^{8kj+2k+4j^2+2j}}{(-q^{2k+2};q^2)_{2j+1}}\left((4j+2)(1+q^{2k+4j+2}-q^{4k+4j+2})-2q^{4k+4j+2}\right)\nonumber\\
	=&\sum_{j=0}^{\infty}\frac{q^{8kj+2k+4j^2+2j}}{(-q^{2k+2};q^2)_{2j+1}}\left((4j+2)+(4j+2)q^{2k+4j+2}(1-q^{2k})-2q^{4k+4j+2}\right)\nonumber\\
	=&\sum_{j=0}^{\infty}\frac{q^{8kj+2k+4j^2+2j}}{(-q^{2k+2};q^2)_{2j+1}}\left(4j+2(1-q^{4k+4j+2})+(4j+2)q^{2k+4j+2}(1-q^{2k})\right)\label{6n+1c2huajian}.
\end{align}
Thus 
\begin{align}
	&\frac{C_4}{(q^2;q^2)_{\infty}^{3}(q^2;q^4)_{\infty}^{2}}\nonumber\\
	=&\frac{1}{(q^2;q^2)_{\infty}^{3}(q^2;q^4)_{\infty}^{2}}\sum_{j=0}^{\infty}\frac{q^{8kj+2k+4j^2+2j}}{(-q^{2k+2};q^2)_{2j+1}}\left(4j+2(1-q^{4k+4j+2})+(4j+2)q^{2k+4j+2}(1-q^{2k})\right)\nonumber\\
	=&\frac{1}{(q^2;q^2)_{\infty}(q^2;q^4)_{\infty}^{2}}\sum_{j=0}^{\infty}\frac{q^{8kj+2k+4j^2+2j}}{(-q^{2k+2};q^2)_{2j+1}(q^2;q^2)_{\infty}}\cdot\nonumber\\
	&\left(\frac{4j}{(q^2;q^2)_{\infty}}+\frac{2(1-q^{4k+4j+2})}{(q^2;q^2)_{\infty}}+\frac{(4j+2)q^{2k+4j+2}(1-q^{2k})}{(q^2;q^2)_{\infty}}\right).\label{eq6n+1c4}
\end{align}
Note that 
\begin{equation}\label{661}
	\frac{1}{(-q^{2k+2};q^2)_{2j+1}(q^2;q^2)_{\infty}}=\frac{(-q^{2k+4j+4};q^2)_{\infty}}{(-q^{2k+2};q^2)_{\infty}(q^2;q^2)_{\infty}}=\frac{(-q^{2k+4j+4};q^2)_{\infty}}{(q^2;q^2)_{k}(q^{4k+4};q^4)_{
	\infty}}
\end{equation}
has non-negative coefficients. Moreover,
\begin{equation}\label{662}
	\frac{1-q^{4k+4j+2}}{(q^2;q^2)_{
	\infty}}=\frac{1+q^2+\cdots+q^{4k+4j}}{(q^4;q^2)_{\infty}}
\end{equation}
and 
\begin{equation}\label{663}
	\frac{1-q^{2k}}{(q^2;q^2)_{\infty}}=\frac{1+q^2+\cdots+q^{2k-2}}{(q^4;q^2)_{\infty}}
\end{equation}
are both have non-negative coefficients. Thus combining \eqref{661}, \eqref{662}, \eqref{663}, we conclude that $\dfrac{C_4}{(q^2;q^2)_{\infty}^{3}(q^2;q^4)_{\infty}^{2}}$ has non-negative coefficients. This completes the proof.
\end{proof}

We now prove Lemma \ref{lem6n+1-5}.

\noindent{\it Proof of Lemma \ref{lem6n+1-5}}.
	From Lemma \ref{lem6n+1-1}, it suffices to show both $\dfrac{C}{(q^2;q^2)_{\infty}^{3}(q^2;q^4)_{\infty}^{2}}$ and $\dfrac{D}{(q^2;q^2)_{\infty}^{3}(q^2;q^4)_{\infty}^{2}}$ have non-positive coefficients. Using Lemma \ref{lem6n+1-2} and Lemma \ref{lem6n+1-c1}, we know that $\dfrac{C}{(q^2;q^2)_{\infty}^{3}(q^2;q^4)_{\infty}^{2}}$ has non-positive coefficients.
	Now we confirm that $\dfrac{D}{(q^2;q^2)_{\infty}^{3}(q^2;q^4)_{\infty}^{2}}$ also has non-positive coefficients. Actually,
	\begin{align}
		\frac{D}{(q^2;q^2)_{\infty}^{3}(q^2;q^4)_{\infty}^{2}}	=&\frac{\sum_{n=0}^{\infty}(6k-3)q^{3n^2+(6k+1)n+2k}(q^{4n+4k+2}-1)}{(q^2;q^2)_{\infty}^{3}(q^2;q^4)_{\infty}^{2}}\nonumber\\
		=&-\frac{\sum_{n=0}^{\infty}(6k-3)q^{3n^2+(6k+1)n+2k}(1+q^2+\cdots+q^{4n+4k})}{(q^2;q^2)_{\infty}^{2}(q^4;q^2)_\infty(q^2;q^4)_{\infty}^{2}}	\end{align}
has non-positive coefficients.
	Thus we deduce that $\sum_{n=-k}^{k-1}h_i(n)\le 0$ for any $k\ge 1$ and $i\ge 1$.
\qed

We now in a position to prove the Theorem \ref{thm-3}.

{\noindent \it Proof of Theorem \ref{thm-3}.} Clearly from the definition of $h_i(n)$, we see that $h_{i}(n)\ge 0$ if $n\ge 0$ and $h_i(n)\le 0$ if $n<0$. In addition, from Theorem \ref{thmchan2016truncated6n+1} we know $\sum_{n=-k}^{k}h_i(n)\ge 0$ for any $i\ge 0$ and $k\ge 1$.  Moreover, from Lemma \ref{lem6n+1-5} we see that $\sum_{n=-k}^{k-1}h_i(n)\le 0$ for any $i\ge 1$ and $k\ge 1$. Thus by Lemma \ref{lemqn}, we deduce that $\Sg(a+b)\sum_{a}^{b}h_i(n)\ge 0$ for any $a,b\in \mathbb{Z}$ and $i\ge 1$, which implies that the coefficients of $q^i$ in \eqref{eqatob} are non-negative for any $i\ge 1$. This completes the proof.
\qed

We conclude this paper by proving Corollary \ref{cor}.

{\noindent \it Proof of Corollary \ref{cor}.}
Using the well-known Euler identity (see \cite[pp. 5]{andrews1998theory})
\[(-q;q)_\infty=\frac{1}{(q;q^2)_\infty},\]
and the generating function of overpartitions \cite{corteel2004overpartitions}
\[
\sum_{n=0}^{\infty}\overline{p}(n)q^n=\frac{(-q;q)_{\infty}}{(q;q)_{\infty}},
\]
we deduce that
\begin{equation}\label{eq-ppp-trip}
	\sum_{n=0}^\infty p\overline{pp}(n)q^n=\frac{1}{(q;q)_\infty}
	\frac{(-q;q)^2_\infty}{(q;q)^2_\infty}=\frac{1}{(q;q)_\infty^3(q;q^2)_\infty^2}.
\end{equation}
Set $q^2$ into $q$ in Theorem \ref{thm-3}, we deduce that
\begin{equation}\label{eqatob-1}
	\Sg(a+b)\frac{1}{(q;q)_{\infty}^{3}(q;q^2)_{\infty}^{2}}\sum_{n=a}^{b}(6n+1)q^{\frac{3n^2+n}{2}}
\end{equation}
has non-negative coefficients of $q^k$ for any $k\ge 1$.

Combining \eqref{eq-ppp-trip} and \eqref{eqatob-1}, we conclude that for any $n\ge 0$,
\[\Sg(a+b)\sum_{i=a}^{b}(6i+1)p\overline{pp}(n-\frac{i(3i+1)}{2})\ge 0.\]
\qed

\noindent{\bf Acknowledgments.}   This work was supported by the National Science Foundation of China grants 12171358 and 12371336.

\end{document}